\documentclass[reqno,a4paper]{amsart}

\usepackage{a4wide}
\usepackage[foot]{amsaddr}
\usepackage{graphicx,enumerate,nicefrac,color,bm}
\usepackage{stmaryrd}
\usepackage{float}
\usepackage{amssymb,,amsfonts,amstext}
\usepackage{algorithm,algorithmicx,algpseudocode}
\usepackage{cite}
\usepackage{ulem}
\usepackage{color}
\usepackage{xcolor}
\usepackage{subfigure} 
\usepackage{epstopdf}

\newcommand{\norm}[1]{\left\|#1\right\|}

\newcommand{\F}{\mathsf{F}}

\newcommand{\J}{\mathsf{J}}

\newcommand{\M}{\mathsf{M}}

\newcommand{\A}{\mathsf{A}}

\newcommand{\Id}{\,\mathsf{Id}}

\newtheorem{theorem}{Theorem}[section]

\newtheorem{cor}[theorem]{Corollary}

\theoremstyle{definition}
\newtheorem{example}[theorem]{Example}

\title[A global Newton-type scheme based on a simplified Newton-type approach]{A global Newton-type scheme based on a simplified Newton-type approach}
	
\author[M.~Amrein]{Mario Amrein}
\address{Applied University of Zurich, CH-8400 Switzerland}
\email{mario.amrein@zhaw.ch}

\begin{document}
\normalem
\begin{abstract}
Globalization concepts for Newton-type iteration schemes are widely used when solving nonlinear problems numerically. Most of these schemes are based on a predictor/corrector step size methodology with the aim of steering an initial guess to a zero of $f$ without switching between different attractors. In doing so, one is typically able to reduce the chaotic behavior of the classical Newton-type iteration scheme. In this note we propose a globalization methodology for general Newton-type iteration concepts which changes into a simplified Newton iteration as soon as the transformed residual of the underlying function is small enough. Based on Banach's fixed-point theorem, we show that there exists a neighborhood around a suitable iterate $x_{n}$ such that we can steer the iterates---without any adaptive step size control but using a simplified Newton-type iteration within this neighborhood---arbitrarily close to an exact zero of $f$. We further exemplify the theoretical result within a global Newton-type iteration procedure and discuss further an algorithmic realization. Our proposed scheme will be demonstrated on a low-dimensional example thereby emphasizing the advantage of this new solution procedure.     
\end{abstract}



\maketitle

\section{Introduction}
For the time being, let $U\subset \mathbb{R}^{n}$ be open and $f:U\rightarrow \mathbb{R}^{n}$ be of class $C^{1}(U;\mathbb{R}^{n})$. In this note we are interested in finding the zeros $x\in U$ of $f$ i.e., we aim to solve the equation
\begin{equation}
\label{eq:1} x\in U: \quad f(x)=0.
\end{equation}

In general---apart from trivial toy problems---the solutions $x_{\infty}$ can only be computed numerically. 
Here, we focus on the following approach: 
For $x\in U$ we consider the matrix-valued map $x\mapsto \M(x) \in \mathbb{R}^{n\times n}$ and define 
$\F(x):=-\M(x)^{-1}f(x)$. Supposing that $\M(x)$ is invertible on a suitable subset of $U$, we now concentrate on the initial value problem
\begin{equation}
\label{eq:initial}
\begin{cases}
\dot{x}(t)&=\F(x(t)), \quad t\geq 0,\\
x(0)&=x_0.
\end{cases}
\end{equation}

This initial value problem tackles the problem of finding the zeros of $f$ from a dynamical system approach. In fact, if $\M(x)$ is given by the Jacobian of $f$ we recover the well known continuous Newton scheme formally satisfying $f(x(t))=f(x_0)\mathrm{e}^{-t}$. For an excellent survey of the continuous Newton scheme see e.g. \cite{1,2,3,D04}. Indeed, supposing that a solution $x(t)$ exists for all time $t\geq 0$, we can try to follow the trajectory of $x(t)$ numerically in order to end up with an approximate root for $f$. For an initial guess $x_{0}\in U$ the simplest routine for solving \eqref{eq:initial} numerically is given by the forward Euler method:
\begin{equation}
\label{eq:forward}
x_{n+1}=x_{n}-t_n\M(x_{n})^{-1}f(x_{n}), \quad t_{n}\in (0,1], n\geq 0.
\end{equation} 

For example, if we choose $\M(x):=\Id$, the above iteration scheme is termed \emph{Piccard-Iteration}. If $\J_{f}(x)$ signifies the Jacobian of $f$ at $x\in U$, then for $\M(x)=\J_{f}(x)$ we observe  a \emph{damped Newton-method}. Another well established scheme is given by setting $\M(x):=\J_{f}(x_0)$, which is also called \emph{simplified Newton method}. The last choice simply freezes the information of the Jacobian throughout the whole iteration procedure. This typically reduces the computational effort in each iteration step. On the other hand, the number of iterations increases in general and the domain of convergence is reduced by this method. However, on a local level, i.e., when the initial guess $x_0$ is supposed to be `sufficiently' close to a zero of $f$, it is reasonable to expect that the simplified Newton method safely leads to a zero which is located next to the initial guess $x_0$. Indeed, if the update $\M(x_n)^{-1}f(x_{n})$ is small enough, we will see in Section 2 that there exists a unique zero for $f$ locally that can be obtained by the following simplified Newton-type iteration scheme:
\begin{equation}
\label{eq:simplified-newton}
u_{j+1}=u_{j}-\M(x_{n})^{-1}f(u_{j}), \quad u_{n}=x_{n}, \quad j\geq n. 
\end{equation}

This observation is especially interesting when the computation of the matrix $\M(x_{n})$ is computationally expensive---as for instance when we solve extremly large scale nonlinear problems arising from the discretization of PDE's. Furthermore, the proposed result in this work asserts local uniqueness of the solution. Thus, one can think of steering an initial guess $x_0\in U$ assumed to be far away of a zero for $f$, `sufficiently' close to the root which is located next to $x_{0}$. Having hit the domain of local uniqueness of the underlying zero we then switch from the adaptive iteration \eqref{eq:forward} to the simplified iteration scheme given in \eqref{eq:simplified-newton} without using any adaptive step-size control.

\subsubsection*{Notation:} In this note we signify by $(\cdot,\cdot)$ the standard Euclidean product of $\mathbb{R}^{n}$. For any $x$ its norm is given by $\norm{x}:=\sqrt{(x,x)}$. For a matrix $\M\in \mathbb{R}^{n\times n}$ we further use the operator norm $\norm{\M}:=\sup_{\norm{x}=1}{\norm{\M x}}$. By $B_{R}(x)$ we denote the closed ball of radius $R$ centered at $x\in \mathbb{R}^{n}$.
Finally, whenever the function~$f$ is differentiable, the derivative at a point~$x\in U$ is written as~$\J_{f}(x)$, thereby referring to the Jacobian of~$f$ at~$x$.

\subsubsection*{Outline:} This note is organized as follows: In section \ref{sec:2} we state and prove a convergence result for a general class of simplified Newton-type iterations schemes as given in \eqref{eq:simplified-newton}. Therefore we firstly discuss the assumptions that have to hold true in order to establish the proposed convergence result. In particular, we embed the local convergence result into a global---and therefore adaptive---Newton-type iteration scheme as given in \eqref{eq:forward}. On that account, in section \ref{sec:3} we finally present and discuss our adaptive strategy on a low dimensional example employing the advantage of the proposed iteration scheme. In section \ref{sec:4} we summarize and comment our findings.

\section{A convergence result}
\label{sec:2}

As a preparation towards the proposed main result we firstly address the assumptions that have to hold. In addition, we comment on a possible extension of the proposed result to a general Banach space framework.

\subsection{Assumptions:}

Suppose we are given an initial value $x_{0}\in U$ and suppose we can compute 
\begin{equation}
\label{eq:Newton-Like}
x_{j+1}=x_{j}-t_{j}\M(x_{j})^{-1}f(x_{j}), \quad t_{j}\in (0,1], \quad j\geq 0.
\end{equation}
Here, $t_{j}$ signifies some adaptively chosen step size (see, e.g. \cite{Amrein:18,Potschka,AmreinWihler:14,ScWi11}).

Let $U$ be an open and convex subset of $\mathbb{R}^{n}$ and assume further that there exists an iterate $x_{n}\in U$ such that there holds the following assumptions:

\begin{itemize}
\item[\textbf{A1.}] Let $\omega $ be a positive constant. For any $v \in U$ and for any $z \in \{tx_n+(1-t)v|t\in [0,1]\}$ we assume that there holds the following affine covariant type Lipschitz-condition on $\J_{f}$:
\begin{equation}
\label{eq:affine invariant}
\norm{\M(x_n)^{-1}(\J_{f}(x_n)-\J_{f}(z))(x_n-v)}\leq \omega (1-t)\norm{x_n-v}^{2}.
\end{equation}
\item[\textbf{A2.}]
We further need $\M(x_n)^{-1}$ to be a sufficiently accurate approximate of the inverse of the Jacobian $\J_{f}(x_n)$ which we here quantify by the following assumption
\[
\norm{\Id-\M(x_n)^{-1}\J_{f}(x_n)}\leq \kappa<1.
\]
\item[\textbf{A3.}]
For $\alpha_{n}:=\norm{\M(x_{n})^{-1}f(x_{n})}$ we need to assume that 
\begin{equation}
\label{eq:main-condition}
\omega \alpha_{n} \leq \frac{(1-\kappa)^2}{2}.
\end{equation}
\item[\textbf{A4.}] For 
\begin{equation}
\label{eq:radius}
R:=\frac{1-\kappa}{\omega}+ \sqrt{\frac{(1-\kappa)^2}{\omega^2}-\frac{2\alpha_n}{\omega}}
\end{equation}

there holds $B_{R}(x_{n})\subset U$.
\end{itemize}

Assumption \textbf{A1} is called affine covariant type Lipschitz condition because in case of $\M(x_n)=\J_{f}(x_n)$ the Lipschitz constant $\omega$ is an affine invariant quantity. Indeed, for $\A\in \text{Gl}(n)$ and $\F(x):=\A f(x)$ there holds
\[
\J_{\F}(x_{n})^{-1}(\J_{\F}(z)-\J_{\F}(x_{n}))(v-x_{n})=\J_{f}(x_{n})^{-1}(\J_{f}(z)-\J_{f}(x_{n}))(v-x_{n}).
\]

For further details concerning affine invariance principles within the framework of Newton-type iterations schemes we refer to the excellent monograph \cite{D04} and the proposed adaptive schemes therein.  

Supposing that $\M(x_{n})^{-1}$ is bounded, then condition \eqref{eq:main-condition} in \textbf{A3} also holds true whenever the residual $\norm{f(x_{n})}$ is `sufficiently' small in the sense that 
\begin{equation}
\label{eq:residual}
\omega \alpha_{n}\leq  \omega \norm{\M(x_{n})^{-1}}\norm{f(x_{n})} \leq \frac{(1-\kappa)^2}{2}.
\end{equation}

Thus, the proposed result implies that whenever the norm of the residual $\norm{f(x_{n})}$ is small enough, there exists a zero on a local level. This is of particular interest when solving nonlinear differential equations numerically  within the context of a fully adaptive iteration scheme. More precisely, let $X$ denote a Banach space---in most cases $X=H_{0}^{1}(\Omega)$---and $X'$ its dual respectively. Then the weak formulation of a nonlinear differential equation reads as follows:

Find $x\in X$ such that there holds
\begin{equation}
\label{eq:Banach}
\left\langle f(x),v\right\rangle_{X'\times X}=0 \quad \forall v\in X, \text{i.e.} \quad f(x)=0 \quad 	\text{in} \quad X', 
\end{equation} 
with $\left\langle \cdot, \cdot \right\rangle_{X'\times X}$ signifying the duality pairing in $X'\times X$.

Solving \eqref{eq:Banach} within the context of an adaptive solution procedure over some finite dimensional space $X_{h}\subset X$---here $h$ typically signifies the mesh-size parameter in the finite element method---one then can try to derive computational quantities $\eta_{\textbf{h}}(x_{h})$ and $\eta_{\textbf{L}}(x_{h})$ such that there holds:
\begin{equation}
\label{eq:residualerror}
\norm{f(x_{h})}_{X'}\leq \eta_{\textbf{h}}(x_{h})+\eta_{\textbf{L}}(x_{h}).
\end{equation}
Here, the quantity $\eta_{\textbf{L}}(x_{h})$ signifies an error estimate which measures the linearization error whereas $\eta_{\textbf{h}}(x_{h})$ represents the discretization error (see, e.g. \cite{CongreveWihler:15,HeidWihler:18,HeidWihler:19,El-AlaouiErnVohralik:11,ChaillouSuri:06,ChaillouSuri:07,Han:94,AmreinWihler:15,Doerfler,AmreinWihlerMelenk}). 
Using \eqref{eq:residual} and supposing that the quantities $\eta_{\textbf{h}}(x_{h})$ and $\eta_{\textbf{L}}(x_{h})$ are small enough we obtain
\[
\omega \norm{\M(x_{n})^{-1}}_{\mathcal{L}(X',X)}(\eta_{\textbf{h}}(x_{h})+\eta_{\textbf{L}}(x_{h}))\leq \frac{(1-\kappa)^2}{2},
\] 

i.e. the a posteriori existence of the solution is guaranteed. Indeed, the a posteriori existence in numerical computations has been addressed in detail \cite{Ortner:09}---especially in the context of solving semilinear problems.
However, although we discuss and present our adaptive scheme in view of dealing with systems of nonlinear equations over $\mathbb{R}^{n}$, it is noteworthy that the established convergence result also holds true within a general Banach space setting. Indeed, our convergence result can be used to realize a specialization of the recently established adaptive iterative linearized Galerkin methodology (ILG) discussed in \cite{HeidWihler:18,HeidWihler:19}.

\begin{theorem}
\label{theorem:1}
Suppose that $f \in C^{1}(U;\mathbb{R}^{n})$. Further assume that there holds the assumptions $\emph{\textbf{A1}}\&\emph{\textbf{A2}}\&\emph{\textbf{A3}}$ and $\emph{\textbf{A4}}$.

Then the map
\begin{equation}
\label{eq:g}
U\ni v \mapsto g(v):=v-\M(x_n)^{-1}f(v).
\end{equation}
satisfies 
\[
g(B_{R}(x_{n}))\subset B_{R}(x_{n}).
\] 

\end{theorem}

\begin{proof}

First of all we rewrite the function $g$ as follows	
\[
g(v)=x_{n}-\M(x_{n})^{-1}f(x_{n})-\left((x_{n}-v)- \M(x_{n})^{-1}(f(x_n)-f(v))\right)
\]

Let $v\in B_{R}(x_n)$. For $t\in [0,1] $ we define the line segment $z(t):=tx_{n}+(1-t)v \subset B_{R}(x_{n})$ and use the integral form of the mean value theorem
\[
\begin{aligned}
(x_{n}-v)- \M(x_{n})^{-1}(f(x_n)-f(v))&=(x_n-v)-\int_{0}^{1}{\M(x_{n})^{-1}\frac{\mathrm{d}}{\mathrm{d}t}f(z(t))\mathrm{d}t}\\
&=\int_{0}^{1}{(\Id-\M(x_{n})^{-1}\J_{f}(z(t)))(x_n-v)\mathrm{d}t}\\
&=\int_{0}^{1}{\M(x_n)^{-1}(\M(x_{n})-\J_{f}(z(t)))(x_n-v)\mathrm{d}t}\\
&=\int_{0}^{1}{(\Id-\M(x_n)^{-1}\J_{f}(x_n))(x_n-v)\mathrm{d}t}\\
&\quad +\int_{0}^{1}{\M(x_n)^{-1}(\J_{f}(x_n)-\J_{f}(z(t)))(x_{n}-v)\mathrm{d}t}
\end{aligned}
\]
from where we obtain by \textbf{A1}\&\textbf{A2}
\begin{equation}
\label{eq:zw}
\norm{(x_{n}-v)- \M(x_{n})^{-1}(f(x_n)-f(v))}\leq \left(\kappa+\omega \frac{1}{2}\norm{x_n-v}\right)\norm{x_n-v}.
\end{equation}

Thus there holds
\[
\norm{g(v)-x_n}\leq \alpha_{n}+\left(\kappa+\omega \frac{1}{2}\norm{v-x_n}\right)\norm{v-x_n}
\leq \alpha_n+\left(\kappa+\omega \frac{1}{2}R\right)R
=R.
\]
Employing $\textbf{A3}$, this last equality holds true if

\begin{equation}
\label{eq:R2}
R= \frac{1-\kappa}{\omega}\pm \sqrt{\frac{(1-\kappa)^2}{\omega^2}-\frac{2\alpha_n}{\omega}}.
\end{equation}
\end{proof}

Let us go back to \eqref{eq:R2} in the proof. We see that the map $g$ also satisfies
\begin{equation}
\label{eq:second-ball}
g\left(B_{r}(x_{n})\right)\subset B_{r}(x_{n})
\end{equation}
with $r=\frac{1-\kappa}{\omega}-\sqrt{\frac{(1-\kappa)^2}{\omega^2}-\frac{\alpha_{n}}{\omega}}$.

Next we give an existence result addressing the zeros $u\in U$ of $f$.

\begin{cor}
\label{existence}
Assumptions and notations as in the preceding Theorem \ref{theorem:1}. Then, there exists a zero $u \in B_{R}(x_{n}) $ of $f$.
\end{cor}

\begin{proof}
From the proof of Theorem \ref{theorem:1} we have that $g(B_{R}(x_{n}))\subset B_{R}(x_{n})$. Employing Brouwer's fixed point theorem we deduce the existence of a fixed point $u\in B_{R}(x_{n})$ of $g$ which is the asserted zero of $f$. 
\end{proof}

In view of the iteration procedure \eqref{eq:simplified-newton} it would be preferable if we can guarantee its convergence within the ball $B_{R}(x_{n})\subset U$. Indeed, if $g$ from \eqref{eq:g} is a contraction in $B_{R}(x_{n})$ we can conlude the existence of a unique fixed point of $g$ which can be obtained by iterating \eqref{eq:simplified-newton}. In doing so we need to strengthen the assumptions \textbf{A1}\&\textbf{A2}\&\textbf{A3} and \textbf{A4} as follows:

\begin{itemize}
\item[$\textbf{A}^{\bigstar}\bm{1}.$] Let $\omega^{\star}$ be a positive constant. For any $x,v \in U$ and for any $z \in \{tx+(1-t)v|t\in [0,1]\}$ we assume that there holds the following affine covariant type Lipschitz-condition on $\J_{f}$:
\begin{equation}
\label{eq:affine_invariant_2}
\norm{\M(x_n)^{-1}(\J_{f}(x)-\J_{f}(z))(x-v)}\leq \omega^{\star} (1-t)\norm{x-v}^{2}.
\end{equation}

\item[$\textbf{A}^{\bigstar}\bm{2}.$] For any $x \in U$ there holds:
\begin{equation}
\label{eq:kappa2}
\norm{\Id-\M(x_{n})^{-1}\J_{f}(x)}\leq \kappa^{\star} <1.
\end{equation}

\item[$\textbf{A}^{\bigstar}\bm{3}.$]
For $\alpha_{n}=\norm{\M(x_{n})^{-1}f(x_{n})}>0$ we need to assume that 
\begin{equation}
\label{eq:main-condition-star}
\omega^{\star} \alpha_{n} \leq \frac{(1-\kappa^{\star})^2}{2}.
\end{equation}

\item[$\textbf{A}^{\bigstar}\bm{4}.$] For 
\begin{equation}
\label{eq:radius-star}
R^{\star}:=\frac{1-\kappa^{\star}}{\omega^{\star}}+ \sqrt{\frac{(1-\kappa^{\star})^2}{{\omega^{\star}}^2}-\frac{2\alpha_n}{\omega^{\star}}}
\end{equation}
there holds $B_{R^{\star}}(x_n)\subset U$.
\end{itemize}

Note that for $x=x_{n}$ we have $\omega=\omega^{\star}$ and $\kappa=\kappa^{\star}$. Now we are ready to prove the following result:

\begin{theorem}
\label{theorem:2}
Suppose that $f \in C^{1}(U;\mathbb{R}^{n})$. Further assume that there holds the assumption $\emph{\textbf{A}}^{\bigstar}\bm{1}\&\emph{\textbf{A}}^{\bigstar}\bm{2}\&\emph{\textbf{A}}^{\bigstar}\bm{3}\&\emph{\textbf{A}}^{\bigstar}\bm{4}$.

Then the map from \eqref{eq:g} satisfies firtsly
\[
g(B_{R^{\star}}(x_{n}))\subset B_{R^{\star}}(x_{n}).
\] 

and is a contraction on $ B_{R^{\star}}(x_{n})$.

\end{theorem}

\begin{proof}
The first assertion follows from the proof of Theorem \ref{theorem:1} and choosing $x=x_n$. Thus we are left to show that $g$ is a contraction.
Notice that
\[
\begin{aligned}
g(x)-g(y)&=(x-y)-\M(x_{n})^{-1}(f(x)-f(y))\\
&=\int_{0}^{1}{(\Id-\M(x_n)^{-1}\J_{f}(z(t)))(x-y)\mathrm{d}t}\\
&=\int_{0}^{1}{(\Id-\M(x_{n})^{-1}\J_{f}(x))(x-y)\mathrm{d}t}\\
&+\int_{0}^{1}{\M(x_n)^{-1}(\J_{f}(x)-\J_{f}(z(t)))(x-y)\mathrm{d}t}.
\end{aligned}
\]

Thus, for $x,y\in B_{R^{\star}}(x_{n})$ there holds:
\[
\norm{g(x)-g(y)}=\norm{(x-y)- \M(x_{n})^{-1}(f(x)-f(y))}\leq \left(\kappa^{\star}+\omega^{\star} \frac{1}{2}\norm{x-y}\right)\norm{x-y}.
\]

Since $\kappa^{\star}+\omega^{\star} \frac{1}{2}\norm{x-y}\leq \kappa^{\star} +\frac{\omega^{\star}}{2}R^{\star}$ and $R^{\star}<\frac{2(1-\kappa^{\star})}{\omega^{\star}} $, there holds
\[
\kappa^{\star}+\omega^{\star} \frac{1}{2}\norm{x-y}<1,
\]
i.e., we conclude that $g$ is a contraction.
\end{proof}

\begin{cor}
\label{simplified-conv}
Assumptions and notations as in the preceding Theorem \ref{theorem:2}. Then, for any initial value $x_{n}\in U$ the simplified Newton-like iterates \eqref{eq:simplified-newton} remain in $B_{R^{\star}}(x_{n})$ and converge to a unique zero $u_{\infty}\in B_{R^{\star}}(x_{n})$ of $f$. 
\end{cor}

\begin{proof}
From the proof of Theorem \ref{theorem:1} we have that for $j\geq n $ the iterates $u_{j+1}=g(u_{j})$ remain in $B_{R^{\star}}(x_{n})$. Furthermore we have also shown that $g$ is a contraction on $B_{R^{\star}}(x_{n})$. Thus, by Banach's fixed-point theorem we deduce that $\lim_{j\to \infty}{g(u_{j})}=u_{\infty}$ exists, which is the unique zero of $f$ in $B_{R^{\star}}(x_{n})$. 
\end{proof}

\begin{figure}
\includegraphics[width=0.6\textwidth]{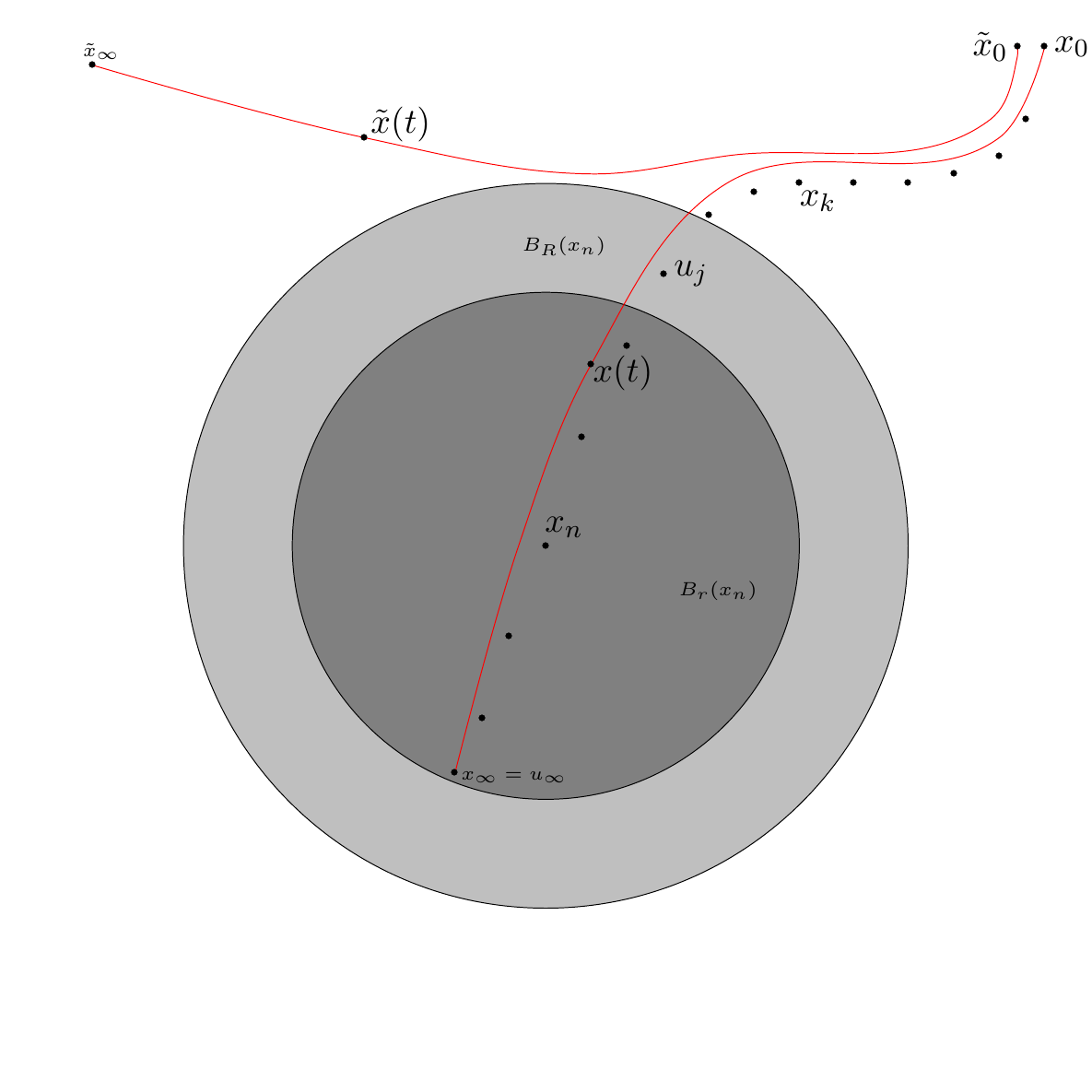}
\caption{The adaptively computed sequence $x_{k}$ switching to the simplified Newton-type scheme within the ball $B_{R}(x_{n})$ which finally leads to the zero $x_{\infty}$. Moreover, we depict two different trajectories $x(t)$ and $\tilde{x}(t)$ respectively---each of them leading to a different zero.}
\label{Fig:Sequence}
\end{figure}

From a computational point of view we can try to switch from the Newton-like iteration scheme \eqref{eq:Newton-Like} to a simplified Newton-like scheme 
\begin{equation}
\label{eq:simplified}
u_{j+1}=u_{j}-\M(x_{n})^{-1}f(u_j), \quad j\geq n
\end{equation}

as soon as there holds $\alpha_{n}\omega^{\star}\leq\frac{(1-\kappa^{\star})^2}{2}$. 
Therefore we need to control the Lipschitz constant $\omega^{\star}$. In doing so, we replace the computational unavailable constant $\omega^{\star}$ by a quantity $\hat{\omega}$ that we can easily compute during the iteration procedure and which comes at no extra cost. Henceforth, suppose we have computed $x_{n+1}, x_{n}$. In view of \eqref{eq:affine invariant}, it is reasonable to switch to the iteration \eqref{eq:simplified} whenever there holds 
\[
\alpha_{n}\hat{\omega}=\alpha_{n}\frac{\norm{\M(x_{n})^{-1}(\J_{f}(x_{n+1})-\J_{f}(x_{n}))(x_{n+1}-x_{n})}}{\norm{x_{n+1}-x_{n}}^2}\leq\frac{(1-\kappa^{\star})^2}{2}<\frac{1}{2}.
\]

In addition, for $\M(x_{n})=\J_{f}(x_n)$ and $x\in B_{R^{\star}}(x_n)$ we observe
\[
\norm{\Id-\J_{f}(x_{n})^{-1}\J_{f}(x)}\approx \norm{\Id-\J_{f}(x_{n})^{-1}\J_{f}(x_{n})}=0,
\]
i.e. $\kappa =0$.

\section{Numerical Experiments}
\label{sec:3}

\subsection{Adaptive strategy}
We now propose a procedure that realizes an adaptive strategy based on the previous observations. The individual computational steps are summarized in Algorithm~\ref{al:full}.

\begin{algorithm}
\caption{Adaptive simplified-Newton-like method:}
\label{al:full}
\begin{algorithmic}[1]
\State{\textbf{Input}:}
\begin{enumerate} 
\item[$\bullet$]
initial value $x_0 \in U$,
%\item[$\bullet$]
%lower bound for the step size $t_{\text{lower}}>0$, 
\item[$\bullet$]
error tolerance $\varepsilon>0$ respectively.
\end{enumerate}
\State{$\delta_0\leftarrow -\M(x_0)^{-1}f(x_0)$} \Comment{compute a first correction}
\State{$t\leftarrow \min\left(1,t\right)$} \Comment{compute an initial step size based on an adaptive procedure, see e.g. \cite{Amrein:18,Potschka, ScWi11}}
\State{$x_s \leftarrow x_0$}
\For{$k=1,2,\dots$}
\If{$\norm{\delta_0}\leq \varepsilon$}
\State{\Return $x_{\infty} \leftarrow x_0$}{\Comment return the solution}
\Else
%\Loop \Comment{start the adaptive procedure}
%\If{$t<t_{\text{lower}}$}
%\State{\textbf{stop the iteration}}\Comment{The adaptive procedure is classified as not convergent}
%\Else
\State{$t\leftarrow t$}\Comment{Compute a step size based on an adaptive procedure, see e.g.\cite{Amrein:18,Potschka,ScWi11}}
\State{{$x_0\leftarrow x_0+t\delta_0$}{\Comment perform a step}}
%\State{\textbf{break}}{\Comment If the adaptive procedure realizes a step size $t$, break the loop}
%\EndIf
%\EndLoop
%%%%%%%%%%%%%%%%%
\State{$\omega\leftarrow \frac{\norm{\M(x_s)^{-1}(\J_{f}(x_0)-\J_{f}(x_s))(x_0-x_s)}}{\norm{x_0-x_s}^{2}}$} \Comment{compute the Lipschitz constant}
\If{$\norm{\delta_0}\omega \leq\frac{1}{2}$}{\Comment start the simplified Newton-like scheme}
\State{Compute $x_{\infty}$ based on the simplified iteration scheme \eqref{eq:simplified-newton}}
\State{\Return $x_{\infty}$}{\Comment return the solution}
\State{\textbf{break the iteration}}
\EndIf
%%%%%%%%%%%%%
\State{{$\delta_0\leftarrow -\M(x_0)^{-1}f(x_0)$}{\Comment update the direction}}
\State{{$t\leftarrow \min{\left(1,t\right)}$}{\Comment predict the step size}}
\State{$x_s \leftarrow x_{0}$}
\EndIf
\EndFor
\end{algorithmic}
\end{algorithm}

Let us briefly comment on the adaptive procedure \ref{al:full}: 
In steps $3 \& 18$ we predict a step size $t$ such that $t=1$ whenever the iterates are `close enough' to the zero $x_{\infty}$. Thus, the proposed procedure allows full steps whenever the iterates are `sufficiently' close to $x_{\infty}$. The computation of $t\in (0,1]$ typically relies on a computational upper bound with respect to the distance $\norm{x(t_{n})-x_{n}}$. There exists different suggested approaches towards an effective computation of the step size $t$ (see e.g., \cite{D04,HeidWihler:18,Amrein:18,Potschka,AmreinWihler:14,AmreinWihler:15,ScWi11}). Here we use the adaptive step size control given in \cite{Amrein:18}.    

This adaptive choice of the step size $t$ consists mainly of two parts: A prediction for the step size $t$ and a correction of the step size whenever $\norm{x_{n}-x(t_{n})}> \tau$. Here, $x(t)$ signifies the exact trajectory leading to the zero $x_{\infty}$ and $x_{n}$ is the numerical solution. Thus, the input $\tau$ is a parameter that determines how close the iterates $x_{n}$ tracks the exact trajectory $x(t)$ leading to a zero of $f$ (see also Figure \ref{Fig:Sequence}). For $\tau=\infty$ there is no restriction on $x_{n}$, i.e. Algorithm \ref{al:full} reproduces the classical Newton scheme---apart from the simplified Newton scheme given in step 13.
Furthermore, the adaptive scheme from \cite{Amrein:18} needs a lower bound $t_{\text{lower}}$ for the step size $t_{n}$ in \eqref{eq:forward}. Indeed, if $t_{n}$ degenerates to $0$, the iterative scheme is not well defined in the sense that it must be classified as not convergent. 
However,  $\tau$ is an error tolerance used in the proposed adaptive computation of the step size $t$ and determines the distance between the numerically computed iterates and the exact trajectory.

\begin{example}
\label{ex:1}
In this example we choose $\M(x)=\J_{f}(x)$. Let us consider the function
\[
f:\mathbb{C}\rightarrow \mathbb{C}, \quad z \mapsto f(z):=z^6-1.
\]
Here, we identify~$f$ in its real form in~$\mathbb{R}^{2}$, i.e., we separate the real and imaginary parts. The six zeros are given by
\[
Z_{f}=\{(1,0),(\nicefrac{1}{2},\nicefrac{\sqrt{3}}{2}),(-\nicefrac{1}{2},\nicefrac{\sqrt{3}}{2}),(-1,0),(-\nicefrac{1}{2},-\nicefrac{\sqrt{3}}{2}),(\nicefrac{1}{2},-\nicefrac{\sqrt{3}}{2})\} \subset \mathbb{R}^{2}.
\]

Note that~$\J_{f}$ is singular at~$(0,0)$. Thus if we apply the classical Newton method with
$\F(x)=-\J_{f}(x)^{-1}f(x)$ 
in~\eqref{eq:forward}, the iterates close to~$(0,0)$ cause large updates in the iteration procedure. More precisely, the application of~$\F(x)=-\J_{f}(x)^{-1}f(x)$ is a potential source for chaos near~$(0,0)$. Before we discuss our numerical experiment, let us first consider the vector fields generated by the continuous problem~\eqref{eq:initial}. In Figure~\ref{fig:flows1}, we depict the direction fields corresponding to~$\F(x)=f(x)$ (left) and~$ \F(x)=-\J_{f}(x)^{-1}f(x)$ (right). We clearly see that some elements of $Z_{f}$ are repulsive for~$\F(x)=f(x)$. Moreover, some elements of $Z_{f}$ show a curl. If we now consider~$ \F(x)=-\J_{f}(x)^{-1}f(x)$ the situation is completely different: All zeros are obviously attractive. In this example, we further observe that the vector direction field is divided into six different sectors, each containing exactly one element of~$Z_{f}$. 

Next we visualize the domains of attraction of four different Newton-type iteration schemes. More precisely, we test the following four iteration procedures:  
\begin{enumerate}
\item The proposed procedure given in Algorithm \ref{al:full}, i.e. adaptive step size control---with $\tau=0.01$
---and switching to the simplified Newton scheme which we abbreviate by \textbf{AS}.
\item The proposed procedure given in Algorithm \ref{al:full}, i.e. adaptive step size control---with $\tau=0.01$
---but without switching to the simplified Newton scheme which we abbreviate by \textbf{ANS}.
\item The proposed procedure given in Algorithm \ref{al:full}, without step size control---i.e. $\tau=\infty$---and without switching to the simplified Newton scheme which we abbreviate by \textbf{NANS}. This is simply the classical Newton iteration scheme.
\item The proposed procedure given in Algorithm \ref{al:full}, without step size control---i.e. $\tau=\infty$---but switching to the simplified Newton scheme which we abbreviate by \textbf{NAS}.
\end{enumerate}

In doing so, we compute the zeros of~$f$ by sampling initial values on a~$500\times 500 $ grid in the domain~$[-3,3]^2$ (equally spaced). In Figure~\ref{fig:frac12}, we show the fractal generated by the traditional Newton method \textbf{NANS} (left) as well as the corresponding plot for the combination of the classical Newton method and the simplified Newton method \textbf{NAS} (right). It is noteworthy that the chaotic behavior caused by the singularities of~$\J_{f}$ of the iteration procedure \textbf{NAS} is comparable to \textbf{NANS}. 

In Figure~\ref{fig:frac34} we depict the basins of attraction for the adaptive procedure as proposed in Algorithm \ref{al:full} \textbf{AS} (left) and the iteration procedure \textbf{ANS}. The chaotic behavior caused by the singularities of $\J_{f}$ is clearly tamed---by both adaptive schemes \textbf{AS} and \textbf{ANS}.

Let us finally consider some performance data given in Table~\ref{tab:1}. An initial value $x_{0}\in [-3,3]^{2}$ is called convergent if it is in fact convergent and additionally approaches the `correct' zero, i.e. the zero that is located in the same exact attractor as the initial guess $x_{0}$. Table~\ref{tab:1} nicely demonstrates that---in contrast to the non adaptive schemes \textbf{NANS} and \textbf{NAS}---the number of convergent iterations for the adaptive procedures \textbf{AS} and \textbf{ANS} is close to $100\%$.
The second line in Table~\ref{tab:1} shows the computational time---by sampling the computational time for all tested initial guesses $x_{0}\in [-3,3]^{2}$---with respect to the classical Newton iteration scheme \textbf{NANS}, i.e., we depict the quantity 
\[
\frac{\text{computational time of the considered iteration scheme}}{\text{computational time of \textbf{NANS}}}.   
\]

In view of this quantity, the proposed iteration scheme \textbf{AS} is the clear winner compared to \textbf{ANS} as can be seen from line 2 in Table~\ref{tab:1}.

\begin{figure}
\includegraphics[width=0.45\textwidth]{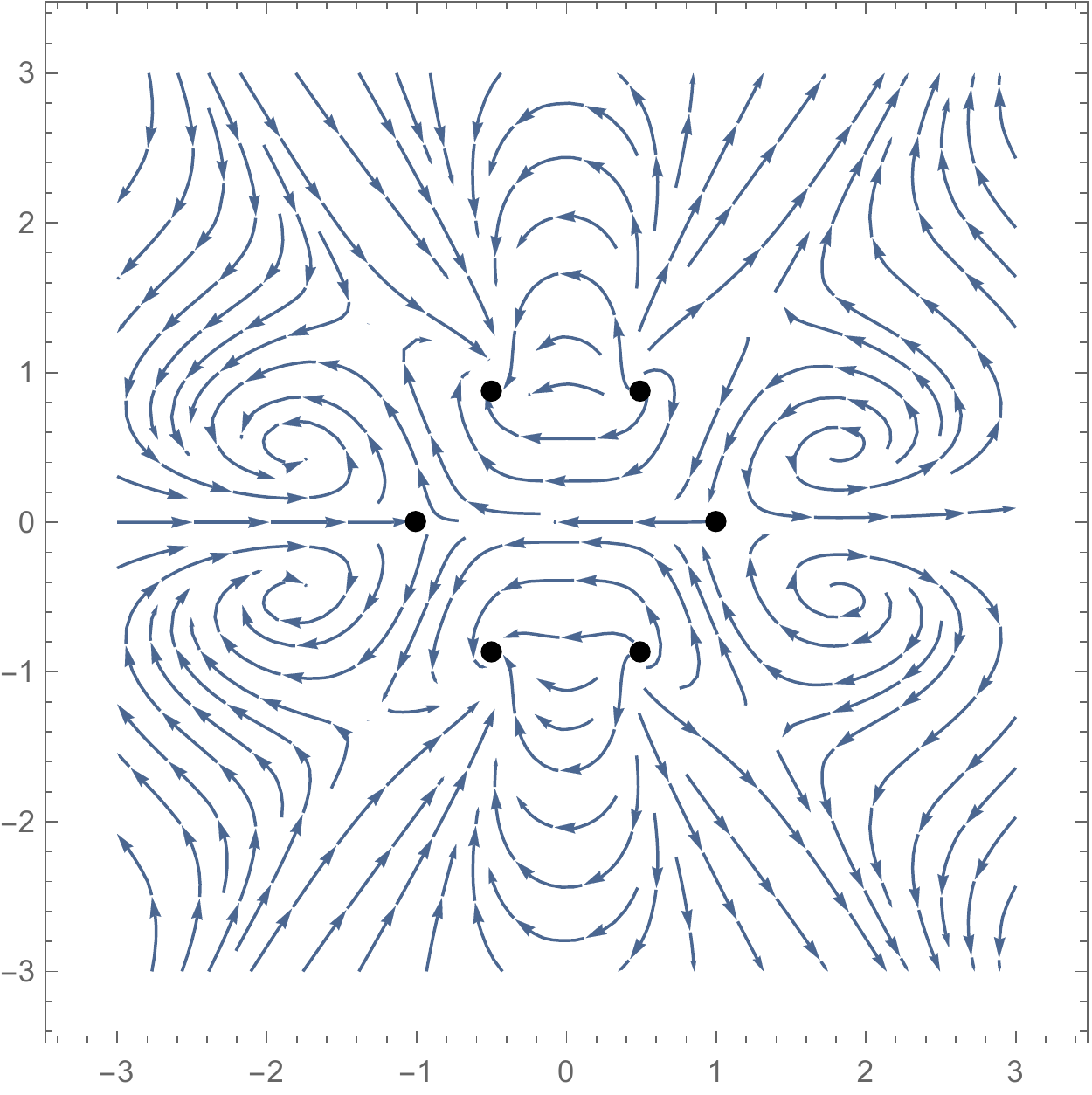}
\hfill
\includegraphics[width=0.45\textwidth]{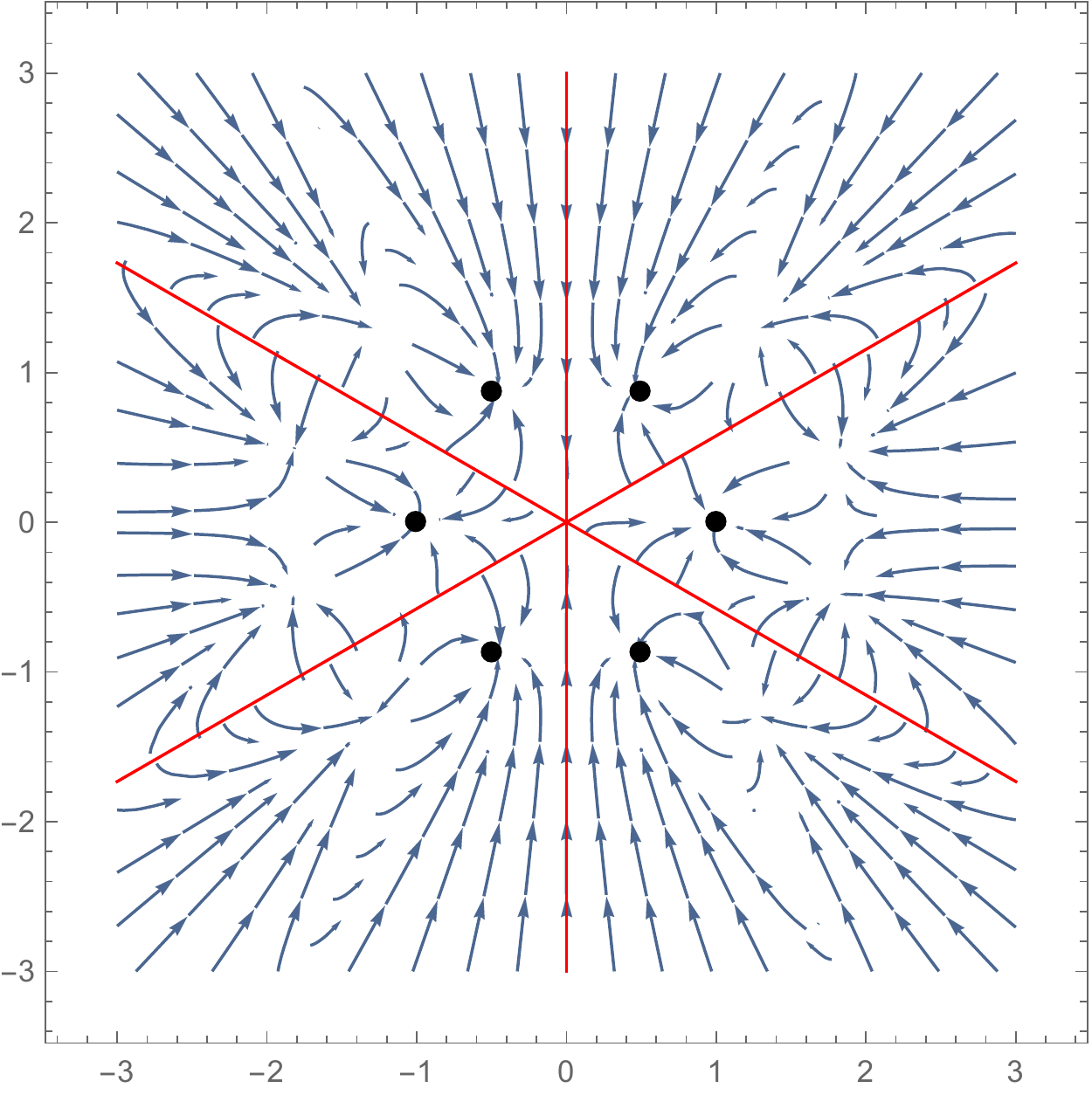}
\caption{Example~\ref{ex:1}: The direction fields corresponding to~$f(z)=z^6-1$ (left) and to the transformed~$\F(z)=-\J_{f}(z)^{-1}\cdot f(z)$ (right).}
\label{fig:flows1}
\end{figure}

\begin{figure}
\includegraphics[width=0.45 \textwidth]{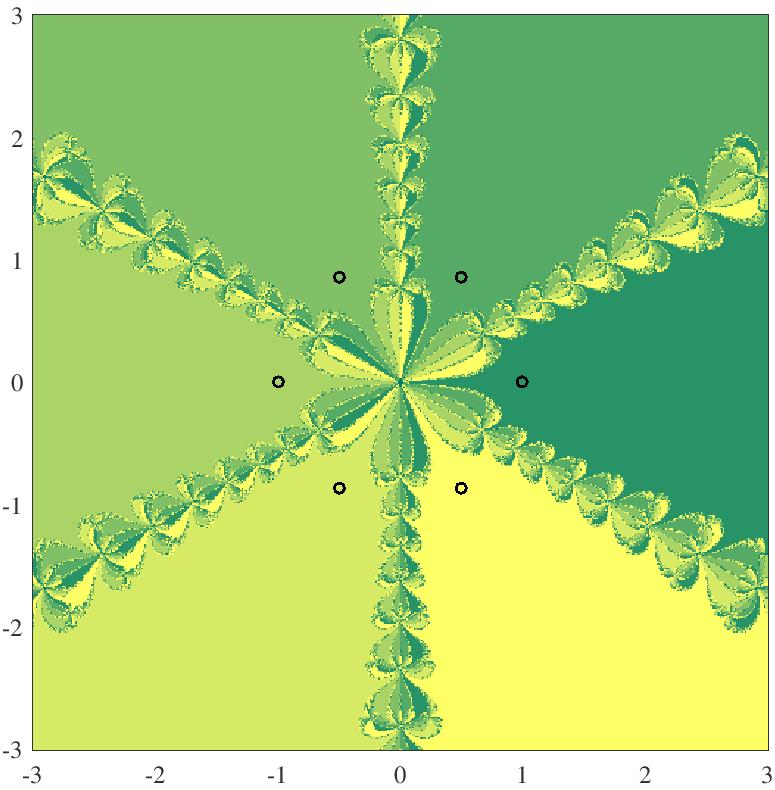}
\hfill
\includegraphics[width=0.45	\textwidth]{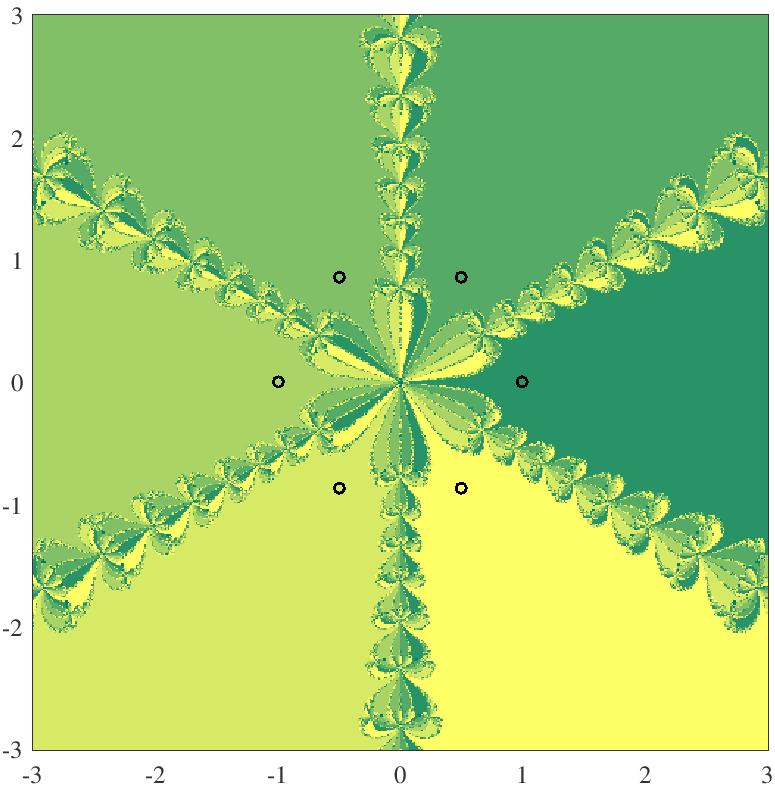}
\caption{The basins of attraction for Example~\ref{ex:1} by the Newton method. On the left for the classical Newton scheme \textbf{NANS} and on the right using the proposed simplified Newton iteration scheme \textbf{NAS}. Different colors distinguish the six basins of attraction associated with the six solutions (each of them is marked by a small circle).}
\label{fig:frac12}
\end{figure}

\begin{figure}
\includegraphics[width=0.45 \textwidth]{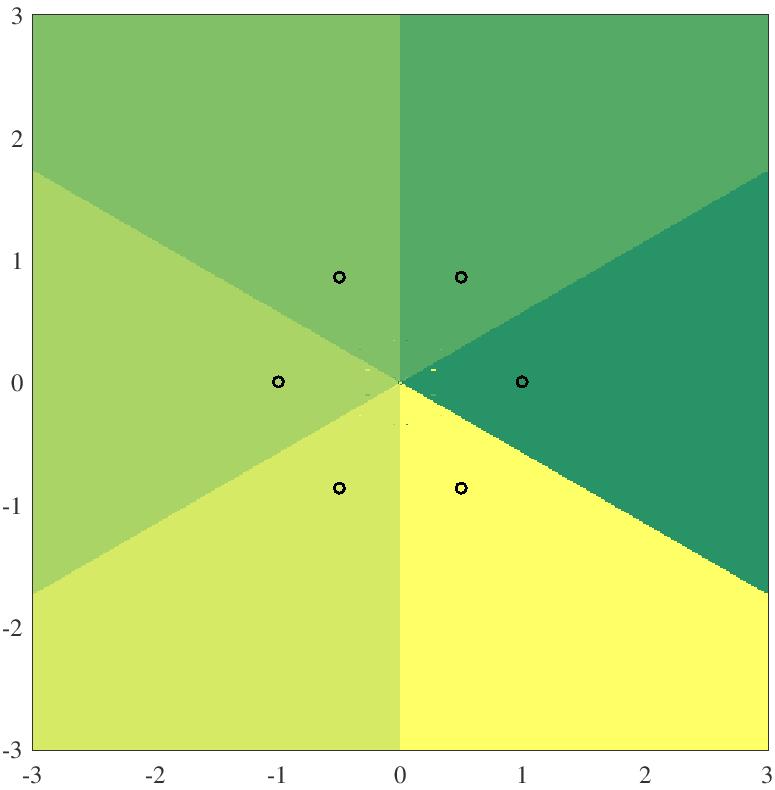}
\hfill
\includegraphics[width=0.45	\textwidth]{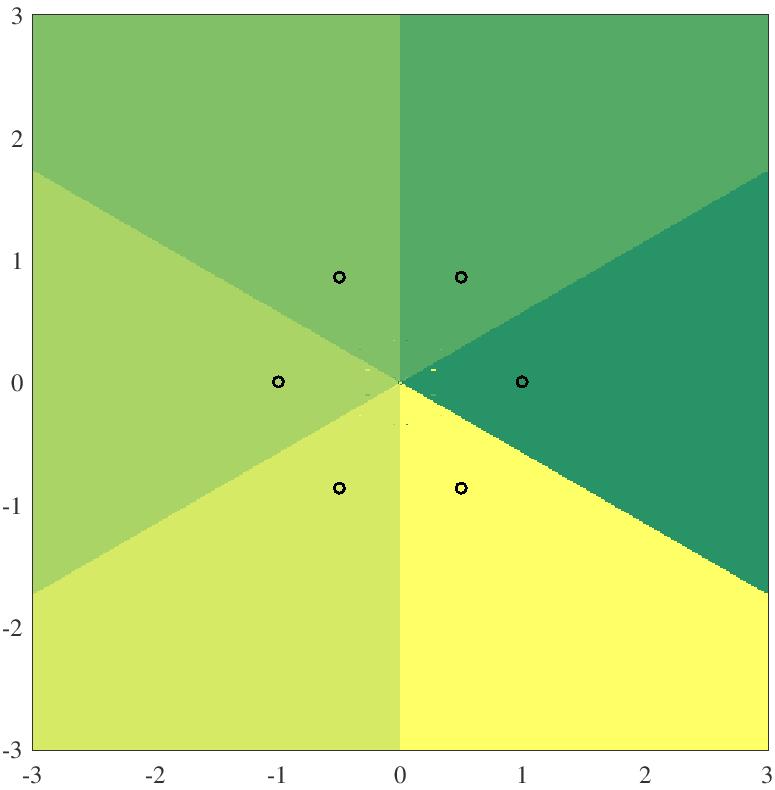}
\caption{The basins of attraction for Example~\ref{ex:1} by the Newton method. On the left with step size control (i.e.,~$t\in(0,1]$) and the proposed scheme based on Algorithm \ref{al:full} \textbf{AS}.  On the right again with step size control (i.e.,~$t\in(0,1]$) but without the simplified scheme---i.e., the derivative $\J_{f}$ was updated in each iteration step \textbf{ANS}. Six different colors distinguish the six basins of attraction associated with the six solutions (each of them is marked by a small circle).}
\label{fig:frac34}
\end{figure}

\begin{table}
\begin{center}
\begin{tabular}{p{2cm} l l l l}
& \textbf{AS} & \textbf{ANS} & \textbf{NANS} & \textbf{NAS} \\
\hline
convergent: & $99.98\%$ & $99.98\%$ & $80.4\%$ & $80.4\% $  \\
\hline 
complexity:  & $2.25$  & $2.7$ & $1$ & $1.03$ \\
\hline
\end{tabular}
\vspace{0.5cm}
\caption{Performance for Examples ~\ref{ex:1}. Here we clearly see the advantage of the proposed adaptive procedure based on the simplified Newton-type scheme \textbf{AS}. This is due to fixed derivative $\M(x_n)=\J_{f}(x_n)$ as soon as $\alpha_{n}\omega\leq \nicefrac{1}{2}$. Furthermore, we see that almost all tested initial guesses are converging to the correct zero.} 
\label{tab:1}
\end{center}
\end{table}

\end{example}

\section{Conclusions}
\label{sec:4}
In this work, we have proved a convergence result for general simplified Newton-type iteration schemes under quite reasonable assumptions. In particular, we have shown that whenever the correction $\norm{\M(x_{n})^{-1}f(x_{n})}$ is small, then there locally exists a unique zero for the underlying map $f$. Since the proof of the proposed result relies on Banach's fixed-point theorem, the theoretical result is constructive in the sense that it can be used for the numerical computation of the locally unique fixed point and therefore of the zero to be considered. Moreover, we have combined the convergence result with an adaptive root finding procedure thereby firstly taming the chaotic behavior of classical Newton-type iteration schemes and secondly reducing the computational effort due to the constant map $x\mapsto \M(x_{n})\in \mathbb{R}^{n\times n}$---without reducing the domain of convergence. We have tested our method on a low dimensional problem. Moreover, our experiment demonstrates empirically that the proposed scheme is indeed capable to tame the \emph{chaotic} behavior of the iteration compared with the classical Newton scheme, i.e., without applying any step size control. In particular, our test example illustrate that the domains of convergence can---typically---be considerably enlarged in the sense that almost all initial guesses $x_{0}$ are convergent to the `correct' zeros---i.e., the zero which is located in the same attractor as the initial guess $x_0$.

\subsection*{Acknowledgement}
The author is grateful to Pascal Heid for comments on an earlier draft of this manuscript.

\bibliographystyle{amsplain}
\bibliography{references}

\providecommand{\bysame}{\leavevmode\hbox to3em{\hrulefill}\thinspace}
\providecommand{\MR}{\relax\ifhmode\unskip\space\fi MR }
\providecommand{\MRhref}[2]{%
  \href{http://www.ams.org/mathscinet-getitem?mr=#1}{#2}
}
\providecommand{\href}[2]{#2}
\begin{thebibliography}{10}

\bibitem{Amrein:18}
M.~Amrein, \emph{Adaptive {N}ewton-type schemes based on projections}, Tech.
  Report 1809.04337v2, arxiv.org, 2018.

\bibitem{AmreinWihlerMelenk}
M.~Amrein, J.~M. Melenk, and T.~P. Wihler, \emph{{A}n hp-adaptive
  {N}ewton-{G}alerkin finite element procedure for semilinear boundary value
  problems}, Mathematical Methods in the Applied Sciences \textbf{40} (2016),
  no.~6, 1973--1985, 13 pages, mma.4113.

\bibitem{AmreinWihler:14}
M.~Amrein and T.~P. Wihler, \emph{{An adaptive {Newton}-method based on a
  dynamical systems approach}}, Communications in Nonlinear Science and
  Numerical Simulation \textbf{19} (2014), no.~9, 2958--2973.

\bibitem{AmreinWihler:15}
\bysame, \emph{Fully {A}daptive {N}ewton--{G}alerkin methods for semilinear
  elliptic partial differential equations}, SIAM, {J}ournal of {S}cientific
  {C}omputing \textbf{37} (2015), no.~4, A1637--A1657, 21 pages.

\bibitem{ChaillouSuri:07}
A.~Chaillou and M.~Suri, \emph{{A posteriori estimation of the linearization
  error for strongly monotone nonlinear operators}}, Journal of Computational
  and Applied Mathematics \textbf{205} (2007), no.~1, 72--87.

\bibitem{ChaillouSuri:06}
A.~L. Chaillou and M.~Suri, \emph{{Computable error estimators for the
  approximation of nonlinear problems by linearized models}}, Computer Methods
  in Applied Mechanics and Engineering \textbf{196} (2006), no.~1-3, 210--224.

\bibitem{CongreveWihler:15}
S.~Congreve and T.~P. Wihler, \emph{Iterative {G}alerkin discretizations for
  strongly monotone problems}, Journal of Computational and Applied Mathematics
  \textbf{311} (2017), 457--472.

\bibitem{D04}
P.~Deuflhard, \emph{{N}ewton methods for nonlinear problems}, Springer Series
  in Computational Mathematics, Springer Verlag Berlin Heidelberg, 2004.

\bibitem{Doerfler}
W.~D{\"o}rfler, \emph{A robust adaptive strategy for the nonlinear {P}oisson
  equation}, Computing \textbf{55} (1995), no.~4, 289--304.

\bibitem{El-AlaouiErnVohralik:11}
L.~El~Alaoui, A.~Ern, and M.~Vohral{\'\i}k, \emph{{Guaranteed and robust a
  posteriori error estimates and balancing discretization and linearization
  errors for monotone nonlinear problems}}, Computer Methods in Applied
  Mechanics and Engineering \textbf{200} (2011), no.~37-40, 2782--2795.

\bibitem{Han:94}
W.~Han, \emph{{A posteriori error analysis for linearization of nonlinear
  elliptic problems and their discretizations}}, Mathematical Methods in the
  Applied Sciences \textbf{17} (1994), no.~7, 487--508.

\bibitem{HeidWihler:18}
P.~Heid and T.P. Wihler, \emph{Adaptive iterative linearization {G}alerkin
  methods for nonlinear problems}, Tech. Report 1808.04990, arxiv.org, 2018.

\bibitem{HeidWihler:19}
P.~Heid and T.P Wihler, \emph{On the convergence of adaptive iterative
  linearized {G}alerkin methods}, Tech. Report 1905.06682, arxiv.org, 2019.

\bibitem{1}
J.W. Neuberger, \emph{Continuous {N}ewton's method for polynomials}, Math.
  Intel \textbf{21} (1999), 18--23.

\bibitem{2}
\bysame, \emph{Integrated form of continuous {N}ewton's method}, Lecture notes
  in Pure and applied. math. \textbf{234} (2003), 331--336.

\bibitem{3}
\bysame, \emph{The continuous {N}ewton's method, inverse functions and {N}ash
  {M}oser}, Amer. Math. Monthly \textbf{114} (2007), 432--437.

\bibitem{Ortner:09}
C.~Ortner, \emph{A posteriori existence in numerical computations}, SIAM J.
  Numer. Anal. \textbf{47} (2009), no.~4, 2550--2577.

\bibitem{Potschka}
Andreas Potschka, \emph{Backward step control for global {N}ewton-type
  methods}, SIAM J. Numer. Anal. \textbf{54} (2016), no.~1, 361--387.
  \MR{3459978}

\bibitem{ScWi11}
H.~R. Schneebeli and T.~P. Wihler, \emph{The {N}ewton-{R}aphson method and
  adaptive {ODE} solvers}, Fractals \textbf{19} (2011), no.~1, 87--99.

\end{thebibliography}
\end{document}